\documentclass[reqno,11pt]{amsart}

\usepackage{amssymb,amsmath,amsthm,amsfonts,xcolor} 

\usepackage{comment,graphicx}
\usepackage{textcomp}
\usepackage{enumitem}

\usepackage[colorlinks=true, pdfstartview=FitV, linkcolor=blue, citecolor=blue, urlcolor=blue]{hyperref}

\usepackage{color,setspace,wrapfig,multicol}

\usepackage[labelformat=empty]{subcaption}

\usepackage[mathscr]{eucal}

\usepackage{lipsum}

\usepackage[left=3cm,right=3cm,top=2.5cm,bottom=2.5cm]{geometry}

\allowdisplaybreaks
\numberwithin{equation}{section}

\newtheorem{theorem}{Theorem}[section]
\newtheorem{proposition}[theorem]{Proposition}
\newtheorem{lemma}[theorem]{Lemma}
\newtheorem{corollary}[theorem]{Corollary}
\newtheorem{remark}[theorem]{Remark}

\theoremstyle{definition}
\newtheorem{definition}[theorem]{Definition}
\newcommand{\be}{\mathbf{e}}

\newcommand{\cC}{\mathcal{C}}
\newcommand{\cF}{\mathcal{F}}
\newcommand{\cH}{\mathcal{H}}

\newcommand{\N}{\mathbb{N}}
\newcommand{\R}{\mathbb{R}}

\newcommand{\sU}{\mathscr{U}}

\newcommand{\dist}{\mathrm{dist}}
\newcommand{\sdist}{\mathrm{sdist}}
\newcommand{\intpart}[1]{\left\lfloor #1\right\rfloor}
\renewcommand{\tilde}{\widetilde}
\renewcommand{\d}{\mathrm{d}}
\newcommand{\sd}{\mathrm{sd}}

\newcommand{\p}{\partial}
\newcommand{\pOmega}{\partial^\Omega}
\newcommand{\cl}[1]{\overline{#1}}

\setlist[itemize]{leftmargin=6mm} % Remove the indentation
\begin{document}

\title[Evolution of droplets in $\R^3$]{Consistency of minimizing movements with smooth  mean curvature flow of droplets with prescribed contact angle in $\R^3$}

\author[Sh. Kholmatov] {Shokhrukh Yu. Kholmatov} 
\address[Sh. Kholmatov]{University of Vienna,  Oskar-Morgenstern Platz 1, 1090 Vienna 
(Austria)}
\email{shokhrukh.kholmatov@univie.ac.at}

\keywords{generalized minimizing movements, capillarity functional, droplet, curvature, smooth curvature flow}

\subjclass[2010]{53C44, 49Q20, 35A15,  35D30, 35D35}

\date{\today}

\begin{abstract}
In this paper we prove that in $\R^3$ the minimizing movement solutions for mean curvature motion of droplets, obtained in \cite{BKh:2018}, coincide with the smooth mean curvature flow of droplets with a prescribed (possibly nonconstant) contact angle.
\end{abstract}

\maketitle

\section{Introduction}

Capillary droplets, known for their distinctive behavior resulting from the interplay of surface tension and capillary forces, have attracted considerable interest across a range of scientific and engineering fields, for instance in the study of wetting phenomena, energy minimizing drops and their adhesion properties, as well as because of their connections  with minimal surfaces (see e.g. \cite{AdS:2005,BB:2012,CM:2007,dGBQ:2004,dSGO:2007,Finn:1986}).

In this paper as in \cite{BKh:2018} we are interested in the mean curvature motion of a droplet in a halfspace with a prescribed (possibly nonconstant) contact angle. Such an evolution can be seen as mean curvature flow of hypersurfaces with a prescribed Neumann-type boundary condition. There are quite a few results related to the well-posedness of the classical mean curvature flow with boundary (see e.g. \cite{Huisken:1989,OU:1991III} for mean curvature flow with Dirichlet boundary conditions and \cite{AW:1994,Guan:1996,KKR:1995} for mean curvature flow with Neumann-type boundary conditions). 

The mean curvature evolution of bounded smooth sets (even without boundary conditions) can produce a singularity in finite time. To continue the flow after singularity, several notions of weak solutions have been introduced, see e.g. \cite{ATW:1993,Bellettini:2013,Brakke:1978,ChMNP:2019,ESS:1992,Giga:2006,Ilmanen:1994,LS:1995}, some of which have been extended to the case with boundary conditions: see e.g. \cite{GOT:2021,White:2021} for Brakke flow with Dirichlet and/or dynamic boundary conditions, \cite{GS:1991,HM:2022,KKR:1995} and \cite{BGM:2023} for viscosity flow with Neumann-type and Dirichlet boundary conditions, \cite{BKh:2018,HL:2021} for BV-distributional solutions with Neumann-type boundary conditions and \cite{BKh:2018} and \cite{MT:1999} for minimizing movements with Neumann-type and Dirichlet boundary conditions. 

In the current paper we study a consistency problem for the minimizing movement solution with the smooth mean curvature flow of droplets. Before the stating the main result of the paper, let us recall some definitions. As in \cite{BKh:2018} modelling the regions occupied by droplets by sets of finite perimeter in $\Omega:=\R^2\times(0,+\infty),$ and we introduce the capillary analogue of the Almgren-Taylor-Wang functional 
\begin{equation}\label{eq:capillar_ATW}
\cF_\beta(E;E_0,\tau):=\cC_\beta(E,\Omega)+\frac{1}{\tau}\int_{E\Delta E_0} \d_{E_0}(x)dx,
\end{equation}
where $E,E_0$ are sets of finite perimeter in $\Omega,$ i.e., $E,E_0\in BV(\Omega;\{0,1\}),$ $\tau>0,$
$$
\cC_\beta(E,\Omega):=P(E,\Omega) + \int_{\p\Omega} \beta\chi_Ed\cH^2
$$
is the capillary functional \cite{DPhM:2015,Finn:1986} for some $\beta\in L^\infty(\p\Omega),$ $\d_{E_0}(\cdot):=\dist(x,\Omega\cap\p^*E_0)$ and $\p^*E_0$ is the reduced boundary of $E_0.$ We endow $BV(\Omega;\{0,1\})$ with the $L^1(\Omega)$-convergence.

\begin{definition}[\textbf{Generalized minimizing movement, \cite{DeGorgi:93}}]
$\,$
\begin{itemize}
\item[\rm (a)]  Given $\tau>0,$ a family $\{E(\tau,k)\}_{k\in\N_0}$ is called a (discrete) \emph{flat flow} starting from $E_0$ and associated to $\cF_\beta$ provided that $E(\tau,0):=E_0,$
$$
\cF_\beta(E(\tau,k); E(\tau,k-1),\tau) = \min_{F\in BV(\Omega;\{0,1\})}\,\cF_\beta(F;E(\tau,k-1),\tau).
$$

\item[\rm(b)] A family $\{E(t)\}_{t\in[0,+\infty)}\subset BV(\Omega;\{0,1\})$ is called a \emph{generalized minimizing movement} (shortly, GMM) starting from $E_0$ if there exist a sequence $\tau_i\to0^+$ and flat flows $\{E(\tau_i,\cdot)\}$ starting from $E_0$ such that 
$$
\lim\limits_{i\to+\infty}\,|E(\tau_i,\intpart{t/\tau_i}) \Delta E(t)|=0,\quad t\ge0,
$$
where $\intpart{x}$ is the integer part of $x\in\R.$
\end{itemize}

\noindent
The collection of all GMM starting from $E_0$ and associated to $\cF_\beta$ will be denoted by $GMM(\cF_\beta,E_0).$
\end{definition}

In \cite{BKh:2018} we have established that if $\|\beta\|_\infty<1,$ then $GMM(\cF_\beta,E_0)\ne\emptyset$ for any bounded $E_0\in BV(\Omega;\{0,1\})$ and each GMM is $1/2$-H\"older continuous in time (see \cite[Theorem 7.1]{BKh:2018} and also Theorem \ref{teo:existence_GMM} below). We call the  elements of $GMM(\cF_\beta,E_0)$ a minimizing movement solution for mean curvature flow of droplets starting from $E_0.$

Now consider the regular case. Let $\beta\in C^{1+\alpha}(\p\Omega)$ (for some $\alpha\in(0,1]$) with $\|\beta\|_\infty<1,$ the initial set $E_0$ be bounded and the manifold $\Omega\cap \p E_0$ be a $C^{2+\alpha}$-hypersurface with boundary, satisfying the contact angle condition (the so-called Young's law \cite{DPhM:2015,Finn:1986})
\begin{equation}\label{sahudaza677}
\nu_{E_0}(x)\cdot \be_3 = -\beta\quad\text{on $\p\Omega\cap \cl{\Omega\cap \p E_0},$}
\end{equation}
where $\nu_{E}$ is the outer unit normal to $E$ and $\be_3=(0,0,1)\in\R^3.$ Then in view of \cite[Theorem B.1]{BKh:2018} there exists a unique family $\{E(t)\}_{t\in[0,T^\dag)},$ defined up to the maximal time $T^\dag,$ such that $E(0)=E_0,$ $E(t)$ satisfies the contact angle condition \eqref{sahudaza677} with $E(t)$ in place of $E_0$ and the surfaces $\Omega\cap \p E(t)$ move by their mean curvature (see also Theorem \ref{teo:short_time} below). For simplicity, let us call $\{E(t)\}$ the smooth mean curvature flow starting from $E_0$ with contact angle $\beta.$

Now we are in position to state the main result of the current paper. 

\begin{theorem}[\textbf{Consistency of GMM with smooth mean curvature flow}]\label{teo:consistence}
Let  $\beta\in C^{1+\alpha}(\p\Omega)$ for some $\alpha\in(0,1]$ with $\|\beta\|_\infty<1$ and $E_0$ be a bounded set such that $\cl{\Omega\cap \p E_0}$ is a $C^{2+\alpha}$-manifold with boundary satisfying the contact angle condition \eqref{sahudaza677}. Let $\{E(t)\}_{t\in [0,T^\dag)}$ be the unique mean curvature flow starting from $E_0$ and with contact angle $\beta.$ Then for every $F(\cdot)\in GMM(\cF_\beta,E_0)$ 
\begin{equation}\label{gmm_teng_smooth}
E(t) = F(t)\quad\text{for any $t\in [0,T^\dag).$}
\end{equation}
\end{theorem}

Thus, as in the classical mean curvature flow without boundary \cite{ATW:1993,JN:2023} the minimizing movement solutions for the mean curvature evolution of droplets in $\R^3,$ coincides with the smooth mean curvature flow as long as the latter exists.

To prove Theorem \ref{teo:consistence} we adapt some ideas in \cite[Theorem 7.3]{ATW:1993}. First consider the smooth flow $\{E(t)\}_{t\in [0,T^\dag)}$ of droplets starting from a smooth droplet $E_0.$ Using the uniquness and the stability of the flow, for each $T\in(0,T^\dag)$ we find $\rho,\sigma>0$ and inner and outer barriers $\{G^\pm[r,s,a,t]\}$ for $r\in [0,\rho],$ $s\in[0,\sigma],$ $a\in [0,T)$ and $t\in [a,T],$ consisting of forced mean curvature evolutions starting from a sort of tubular neighborhoods $G_0^\pm[r,s,a]$ of $E(a),$ with slightly perturbed contact angle (see Corollary \ref{cor:time_foliations} and Theorem \ref{teo:short_time}). Next consider any flat flows $\{F(\tau_j,k)\}$ such that 
\begin{equation}\label{ahsz763fff}
F(\tau_j,\intpart{\cdot/\tau_j}) \to F(\cdot)\quad\text{in $L^1(\Omega)$ as $j\to+\infty$}
\end{equation}
for some $F(\cdot)\in GMM(\cF_\beta,E_0).$ We show the existence of $\bar t>0$ depending only on $\rho$ and $\{E(t)\}$ such that if
\begin{equation}\label{ajsdaiu7678}
G_0^-[0,s,t_0]\subset F(\tau_j,k_0)\subset G_0^+[0,s,t_0]
\end{equation}
for some $t_0\in [0,T)$ and $k_0\ge0,$ then
\begin{equation}\label{mubsatz667}
G^-[0,s,t_0,t_0+k\tau_j]\subset F(\tau_j,k_0+k)\subset G^+[0,s,t_0,t_0+k\tau_j]
\end{equation}
for all $0\le k\le \intpart{\bar t/\tau_j}$ (see Lemma \ref{lem:chaklpakes}). In particular, 
assuming $t_0=0$ and letting $j\to+\infty$ in \eqref{rutzr} and using \eqref{ahsz763fff} we get 
$$
G^-[0,s,0,t]\subset F(t)\subset G^+[0,s,0,t],\quad t\in [0,\bar t).
$$
Now using the smooth dependence of $G^\pm$ on its parameters, letting $s\to0^+$ and recalling $G^\pm[0,0,a,t]=E(t)$ for $t\in [a,T]$ we deduce $E(t)=F(t)$ for all $t\in [0,\bar t).$
To show the equality after $\bar t,$ we observe that  the inclusions in \eqref{mubsatz667} yield
\begin{equation}\label{rutzr}
G_0^-[0,4g(2s),t_0 + \bar t]\subset F(\tau_j,k_0 + \bar k_j)\subset G_0^+[0,4g(2s),t_0 + \bar t]
\end{equation}
for some increasing continuous function $g$ with $g(0)=0$ and all $s\in(0,\sigma)$ with $4g(2s)<\sigma,$ where $\bar k_j:=\intpart{\bar t/\tau_j}.$
This replaces \eqref{ajsdaiu7678} with  $s:=4g(2s),$ $t_0:=t_0+\bar t$ and $k_0:=k_0+\bar k_j.$ Thus, applying the above argument again, we deduce $E(t)=F(t)$ for any $t\in [\bar t,2\bar t).$ Since $T$ is finite and $\bar t$ depends only on $\rho$ and $\{E(t)\},$ in at most $\intpart{T/\bar t}+1$ steps we reach $E(t)=F(t)$ for any $t\in [0,T).$

One of the main difficulties here is to show the existence of $\bar t$ and the passage from inclusions in  \eqref{ajsdaiu7678} to the inclusions of the form 
\begin{equation*}%\label{azc76caca}
G_0^-[\rho,s,t_0+k\tau_j]\subset F(\tau_j,k_0+k)\subset G_0^+[\rho,s,t_0+k\tau_j],\quad 0\le k\le \intpart{\bar t/\tau_j}
\end{equation*}
(see the proof of Lemma \ref{lem:chaklpakes}).
In \cite{ATW:1993} this issue was solved by employing the flat flows starting from balls and discrete comparison principles. In our case we still have the comparison principles (see Lemma \ref{lem:compare_setg}), but because of the boundary terms, we cannot use balls, rather we use Winterbottom shapes \cite{Kholmatov:2024,KSch:2024,Maggi:2012} (see Theorem \ref{teo:compare_with_ball}), which are the solution of certain isoperimetric inequality associated to the capillary functional. Hying \eqref{rutzr} and using the comparison of minimizers of $\cF_\beta$ with smooth inner and outer barriers (Lemma \ref{lem:barrier}), as in \cite[Theorem 7.3]{ATW:1993} we deduce \eqref{mubsatz667}.

It is worth to notice that due to the presence of boundaries, the methods of \cite{JN:2023}, which strongly rely on the uniform ball conditions, seem not applicable in our setting.

The paper is organized as follows. In Section \ref{sec:preliminaries} we provide some preliminary definitions and results which will be important in the proof of Theorem \ref{teo:consistence}. Namely, we study the smooth mean curvature evolution of droplets with prescribed contact angle $\beta,$ and its various features such as forced evolution of its small tubular neighborhoods (Theorem \ref{teo:short_time}) and comparison principles (Theorem \ref{teo:comparison}). Moreover, we recall some properties of the minimizers of $\cF_\beta$ from \cite{BKh:2018} (Theorem \ref{teo:propertie_minimizers}), the existence of GMM (Theorem \ref{teo:existence_GMM}) and study the comparison  of flat flows with truncated balls (Theorem \ref{teo:compare_with_ball}). We complete this section with weak comparison properties of inner and outer barriers for minimizers of $\cF_\beta$ (Lemma \ref{lem:barrier}). These results will be the key arguments in the proof of \eqref{gmm_teng_smooth} in the concluding Section \ref{sec:proof_consos}.

\subsection*{Acknowledgements.} 
I acknowledge support from the Austrian Science Fund  (FWF) Lise Meitner Project M2571 and Stand-Alone Project P33716. Also I am grateful to Francesco Maggi for his discussions on the regularity of contact sets of minimizers of the capillary functional, and in particular, showing his paper \cite{DPhM:2017} with Guido De Philippis.

\section{Preliminaries}\label{sec:preliminaries}

\subsection*{Notation}
In this section we introduce the notation and some definitions which will be used throughout the paper.
Unless otherwise stated, all sets we consider are sets of finte perimeter in $\R^3.$ The coordinates $(x_1,x_2,x_3)$ of $x\in\R^3$ are given with respect to the standard basis $\{\be_1,\be_2, \be_3\}.$ By $B_r(x)$  we denote the open ball in $\R^3$ of radius $r>0$ centered at $x.$ The notion $|F|$ stands for the Lebesgue measure of $F\subset\R^3$. 

Throughout the paper we assume 
$$
\Omega:=\R^2\times(0,+\infty),
$$ 
and by $BV(\Omega;\{0,1\})$ we denote the collection of all sets of finite finite perimeter in $\Omega.$ 

By $\beta\in L^\infty(\p\Omega)$ we denote a relative adhesion coefficient of the boundary $\p\Omega=\R^2\times\{0\}$ of $\Omega$ and throughout we assume
\begin{equation}\label{beta_coercive}
\exists \eta\in(0,1/2):\quad \|\beta\|_\infty\le 1-2\eta.
\end{equation}

Given $E\in  BV(\Omega,\{0,1\})$ we denote by 

\begin{itemize}
\item[-] $P(E,U)$ the {\it perimeter} of $E$ in an open set $U\subset \Omega,$

\item[-] $\p^*E$ the reduced boundary of $E,$

\item[-] $\nu_E(x)$ the generalized outer unit normal of $E$ at $x\in \p^*E.$
\end{itemize}
We always assume that every $E$ coincides with its points $E^{(1)}$ of density one so that $\p E=\cl{\p^*E}.$ 
We refer, for instance, to \cite{AFP:2000,Giusti:1984,Maggi:2012} for a more comprehensive information on sets of finite perimeter.

Given $E\in BV(\Omega;\{0,1\}),$ we define the \emph{distance function} from the (reduced) boundary in $\Omega$ as
$$
\d_E(x):=\dist(x,\Omega\cap \p^*E),\quad x\in \cl{\Omega}.
$$
Similarly, we define the \emph{signed distance} as
\begin{equation*} 
\sd_E(x) =  
\begin{cases}
\dist(x,\Omega\cap \p^*E) & x\in \Omega \setminus E,\\
-\dist(x,\Omega\cap \p^*E) & x\in E
\end{cases}
\end{equation*}
for $E\in BV(\Omega;\{0,1\}).$ 
We also write
\begin{equation*}
\pOmega E:= \Omega\cap \p E
\end{equation*}
and
$$
E\prec F\qquad \Longleftrightarrow \qquad E\subset F\quad \text{and}\quad \dist(\pOmega E, \, \pOmega F)>0
$$
for $E,F\in  BV(\Omega;\{0,1\}).$ Note that
\begin{equation}\label{compare_trunc_sdist}
E\subset  F\quad\Longleftrightarrow\quad \sd_E \ge \sd_F \,\,\text{in $\Omega$}
\qquad \text{resp.}\qquad 
E\prec F\quad\Longleftrightarrow\quad \sd_E > \sd_F \,\,\text{in $\Omega.$}
\end{equation}

The following proposition shows the connection between the regular surfaces and distance functions. 

\begin{proposition}\label{prop:regular_distance}
Let $\Gamma$ be a $C^{2+\alpha}$-surface (not necessarily connected, and with or without boundary) in $\Omega$ for some $\alpha\in[0,1]$. Then:

\begin{itemize}
\item[\rm (a)] for any $x\in \Gamma$ there exists $r_x>0$ such that $\Gamma$ divides $B_{r_x}(x)$ into two connected components and $\dist(\cdot,\Gamma)\in C^{2+\alpha}(B_{r_x}(x)\setminus \Gamma);$

\item[\rm(b)] if $\Gamma$ is compact and has no boundary, then $\inf_{x\in \Gamma} r_x >0,$ i.e., the radius $r_x$ in (a) can be taken uniform in $x;$

\item[\rm(c)] if $\Gamma = \pOmega E$ for some $E\subset\Omega,$ then for any $x\in \Gamma$ there exists $r_x>0$ such that $B_r(x)\subset\Omega$ and $\sd_E \in C^{2+\alpha}(B_{r_x}(x)).$
\end{itemize}
\end{proposition}

These assertions are well-known  (see e.g. \cite{DZ:2011}), and can be proven using the local geometry of $\Gamma,$ i.e. passing to the local coordinates. In case of Proposition \ref{prop:regular_distance} (c) we write $\kappa_E:=\kappa_\Gamma$ to denote the mean curvature of $E$ along the boundary portion $\Gamma$ with respect to the unit normal to $\Gamma,$ outer to $E.$ We also set 
$$
\|II_E\|_\infty :=\sup_{x\in \Gamma} \,|II_\Gamma(x)|,
$$
where $II_\Gamma$ is the second fundamental form of $\Gamma.$
In what follows we always assume that the unit normals of $\pOmega E$ are outer to $E$ so that the mean curvature of the boundaries of convex sets are nonnegative. 

\subsection{Smooth mean curvature evolution of droplets} 

In this section we study mean curvature flow of droplets sitting on an inhomogeneous plane. Since we are mainly interested in droplets with a nonempty contact set on $\p\Omega,$ it is natural to restrict ourselves to the ones without connected components not touching to $\p \Omega.$ Such a restriction leads to the following definition.

\begin{definition}[\textbf{Admissibility}]\label{def:admissible_sets}
$\,$
\begin{itemize}
\item[(a)] We say a set $E\subset\Omega$  is \emph{admissible} provided that there exist $\alpha\in(0,1],$ a bounded $C^{2+\alpha}$-open set $\sU\subset\R^2$ and a $C^{2+\alpha}$-diffeomorphism $p\in C^{2+\alpha}(\cl{\sU};\R^3)$ satisfying  
\begin{equation*} 
p[\sU]=\Gamma,\qquad 
p[\p\sU]=\p\Gamma, \qquad 
p\cdot \be_3>0 \,\,\, \text{in $\sU$}
\qquad\text{and}\qquad
p \cdot \be_3=0 \,\,\, \text{on $\p\sU,$}
\end{equation*}
where $\Gamma:=\pOmega E.$ Any such map $p$ is called a \emph{parametrization} of $\Gamma.$ 

\item[(b)] Let $\beta\in C^{1+\alpha}(\p\Omega),$ $\alpha\in(0,1],$ satisfy \eqref{beta_coercive}. We say $E$ is \emph{admissible with contact angle $\beta$} if $E$ is admissible (with the same $\alpha$) and
$$
\nu_E \cdot \be_3 = -\beta \quad\text{on $\p\Omega\cap \cl{\Gamma}$}.
$$
We call that number 
\begin{equation}\label{minheights}
h_E:=\min_{x\in\cl{\Gamma},\,\nu_{E}(x)=x+\be_3}\, \,\,x\cdot \be_3
\end{equation}
the \emph{minimal height} of $E$.
Since $E$ satisfies the contact angle condition, by assumption \eqref{beta_coercive} $h_E>0.$

\item[(c)] Let $Q$ be a compact set in $\R^m$ for some $m\ge1.$ We say a family $\{E[q]\}_{q\in Q}$ of subsets of $\Omega$ is \emph{admissible} if there exist $\alpha\in(0,1],$ a bounded $C^{2+\alpha}$-open set $\sU\subset\R^2$ and a map $p\in C^{2+\alpha,2+\alpha}(Q \times \cl{\sU};\R^3)$ such that $p[q,\cdot]$ is a parametrization of $\pOmega E[q].$  

\item[(d)] We say a family $\{E[q,t]\}_{q\in Q,t\in [0,T)}$ of subsets of $\Omega$ \emph{admissible} if for any $T'\in(0,T)$ there exist  $\alpha\in(0,1],$ a bounded $C^{2+\alpha}$-open set $\sU\subset\R^2$ and a map $p\in C^{2+\alpha,1+\frac{\alpha}{2},2+\alpha}(Q \times [0,T'] \times \cl{\sU};\R^3)$ such that $p[q,t,\cdot]$ is a parametrization of $\pOmega E[q,t].$  

\end{itemize}
\end{definition}

\begin{remark}
$\,$
\begin{itemize}
\item[\rm(a)] By definition, if $E$ is an admissible set, then the $C^{2+\alpha}$-surface $\Gamma:=\pOmega E$ is diffeomorphic to a bounded  smooth open set in $\R^2$ and not necessarily connected (clearly, boundaries of two connected components do not touch). In particular, $\Gamma$ cannot not have ``hanging'' components compactly contained in $\Omega.$ Moreover, its boundary $\p\Gamma$ lies on $\p\Omega$ and the relative interior of $\Gamma$ does not touch to $\p\Omega.$  

\item[\rm(b)] We are slightly abusing the notion ``contact angle'' identifying the (true) contact angle $\theta\in (0,\pi)$ with its cosine $\beta=\cos\theta.$

\item[\rm(c)] When $Q$ is empty in Definition \ref{def:admissible_sets} (d), then we simply write $\{E[t]\}_{t\in [0,T)}$ to denote the corresponding admissible family.
\end{itemize}
\end{remark}

Recall that if $E\subset\R^3$ is a $C^{2+\alpha}$-set without boundary, then for sufficiently small $\rho>0$ the surfaces $\Gamma_r:=\{\sdist(\cdot,\p E)=r\}$ for $r\in (-\rho,\rho)$ foliates the tubular $\rho$-neighborhood of $\Gamma_0:=\p E,$ and the map $r\mapsto \Gamma_r$ smoothly varies. In the next lemma we construct a similar ``foliation'', for admissible sets with a given contact angle. 

\begin{lemma}[\textbf{Foliations}]\label{lem:foliations}
Let $\beta\in C^{1+\alpha}(\p\Omega),$ $\alpha\in(0,1],$ satisfy \eqref{beta_coercive} and $E_0$ be an admissible set with contact angle $\beta$. Then there exist positive numbers $\rho\in(0,1)$ and $\sigma\in (0,\eta),$ depending  only\footnote{We ignore the dependence on $\alpha$ and $\eta.$} on $\|II_{E_0}\|_\infty$ and $h_{E_0}$ (see \eqref{minheights}), and admissible families $\{G_0^\pm[r,s]\}_{(r,s)\in[0,\rho]\times[0,\sigma]}$ such that $G_0^\pm[0,0]=E_0$ and for all $(r,s)\in [0,\rho]\times [0,\sigma]$:

\begin{itemize}[itemsep=3pt]
\item[\rm(a)] $\dist(\pOmega G_0^\pm[r,s], \pOmega E_0)\ge r+s$ and
\begin{align*}
& G_0^-[r,s] \subset E_0 \subset G_0^+[r,s]\\
& \dist(\pOmega G_0^\pm[r,s],\pOmega G_0^\pm[0,s])=r,\\
&\dist(\pOmega G_0^\pm[0,s],\pOmega E_0)=s;
\end{align*}

\item[\rm(b)] $G_0^\pm[r,s]$ is admissible with contact angle $\beta\pm s.$ 

\end{itemize}
\end{lemma}

\begin{proof}
Without loss of generality we assume that $E_0$ is admissible with the same H\"older exponent $\alpha\in(0,1]$ of $\beta.$ We divide the proof into three steps.
\smallskip

{\it Step 1.} We first construct $\sigma>0$ and the sets $G^\pm[0,s]$ for $s\in[0,\sigma].$

Since $\Gamma_0:=\pOmega E_0$ is $C^{2+\alpha}$ up to the boundary, there exists $b>0$ (depending only on the second fundamental form $II_{\Gamma_0}$ of $\Gamma_0$) and a $C^{2+\alpha}$-surface $\tilde \Gamma_0\subset \R^2\times (-b,+\infty)$ with $\p\Gamma_0\subset \{x_3=-b\}$ and $\Gamma_0 \subset \tilde \Gamma_0.$
By Proposition \ref{prop:regular_distance} (a) there exists $\sigma\in (0,\eta/4)$ (depending only on $\|II_{\tilde \Gamma}\|_\infty$ and $h_{E_0}$) such that for any $s\in [-4\sigma,4\sigma]$ the sets 
$$
\tilde \Gamma_s=
\begin{cases}
\{x\in \Omega\setminus \cl{E_0}: \,\,\dist(x,\tilde \Gamma_0)=s\} & s>0,\\
\{x\in \cl{E_0}: \,\,\dist(x,\tilde  \Gamma_0)=s\} & s\le0
\end{cases}
$$ 
are $C^{2+\alpha}$-surfaces with boundary, depending smoothly (at least $C^{2+\alpha}$) on $s.$ 

Note that $\gamma_0:=\p\Gamma_0$ is a finite union of planar $C^{2+\alpha}$-curves. Let $\hat F\subset \p\Omega$ the bounded planar open set  enclosed by $\gamma_0.$ Decreasing $\sigma$ is necessary (depending only on the $L^\infty$-norm of the planar curvatures of $\gamma_0$ and the minimal height $h_{E_0}$) such that for any $s\in[-\sigma,\sigma]$ we may assume that the sets
$$
\gamma_s:=
\begin{cases}
\{z\in \hat F:\,\dist(z,\gamma_0)=-4s\} & s<0,\\
\{z\in \p\Omega\setminus \hat F:\,\dist(z,\gamma_0)=4s\} & s\ge0,
\end{cases}
\qquad\text{and}\qquad 
\zeta_s:=\tilde \Gamma_s\cap 
\{x_3=\sigma\}
$$ 
are a union of $C^{2+\alpha}$-curves, homotopic to $\gamma_0.$ By the $C^{2+\alpha}$-dependence of $\gamma_s$ on $s$ we can find a bounded $C^{2+\alpha}$-open set $\sU\subset\R^2$ and a map $p\in C^{2+\alpha,2+\alpha}([-\sigma,\sigma]\times \p\sU;\R^3)$ such that $p[s,\p\sU] = \gamma_s$ for any $s\in[-\sigma,\sigma].$ 
Now as in \cite[Remark B.2]{BKh:2018} we can extend each $p[s,\cdot]$ as a diffeomorphism to an $\epsilon$-tubular neighborhood $\sU_r^-:=\{u\in \sU:\,\dist(u,\p \sU)\le \epsilon\}$ of $\p\sU$ (for small $\epsilon>0,$ still keeping $C^{2+\alpha}$-regularity both in $s$ and in $u$) such that the surface $p[s,\sU_r^-]$ lies in $\Omega,$ satisfies the contact angle condition with $\beta+s$ along $\gamma_s$ and the distance to $\Gamma_0$ is $\ge 3s.$ 
Let us also parametrize the truncations $\tilde \Gamma_s\cap \cl{\Omega^\sigma}$ of the surfaces $\tilde \Gamma_s$ by some diffeomorphism $p[s,\cdot]:\{u\in\sU:\,\,\dist(u,\p \sU)\ge8\epsilon\}\to\R^3$ (still keeping $C^{2+\alpha}$-regularity in $s\in [-\sigma,\sigma]$), where $\Omega^\sigma:=\R^2\times(\sigma,+\infty)$. 
Now we extend $p$ arbitrarily to $[-\sigma,\sigma]\times \{u\in\sU:\,\,\epsilon\le\dist(u,\p\sU)\le 8\epsilon\}$ in a way that $p\in C^{2+\alpha,2+\alpha}([-\sigma,\sigma]\times\cl{\sU}),$ $p[s,\cdot]$ is a diffeomorphism, $p[0,\sU]=\Gamma_0$ and the distance between surfaces 
$p[s,\cl{\sU}]$ and $\Gamma_0$ is equal to $s.$ 

For $s\in[0,\sigma]$ we denote by $G_0^+[0,s]$ and $G_0^-[0,s]$ the bounded sets enclosed by $\p\Omega$ and the $C^{2+\alpha}$-surfaces $\Gamma_0^+[0,s]:=p[s,\cl{\sU}]$ and $\Gamma_0^-[0,s]:=p[-s,\cl{\sU}],$ respectively.

Notice that by construction $G_0^{\pm}[0,0]=E_0$ and $\dist(\pOmega G_0^\pm[0,s],\pOmega E_0)=s$ for any $s\in[0,\sigma].$ 
\smallskip

{\it Step 2.} Now we construct $\rho>0$ and $G_0^\pm[r,s]$ for $r\in[0,\rho]$ and $s\in[0,\sigma].$ 

Since $\gamma_s$ is $C^{2+\alpha}$-regular in $s\in[-\sigma,\sigma],$ slightly decreasing $\sigma$ if necessary, we find $\rho>0$ depending only on $\sigma,$ $\|II_{E_0}\|_\infty$ and $h_{E_0}$ such that for any $r\in[0,\rho]$ the sets 
$$
\gamma_{r,-s}^- := \{z\in \hat F_s:\,\,\dist(z,\gamma_{s})=4r\}, \quad  s\le0
$$
and
$$
\gamma_{r,s}^+ :=\{z\in \p\Omega\setminus \hat F_s:\,\,\dist(z,\gamma_s)=4r\},\quad  s\ge0
$$
are finite unions of $C^{2+\alpha}$-curves homotopic to $\gamma_s,$ where $\hat F_s\subset\p\Omega$ is a bounded set enclosed by $\gamma_s$. As above consider the truncations $\Gamma_0^\pm[0,s]\cap \cl{\Omega^\sigma}.$ Since these truncations are smooth (at least  $C^{2+\alpha}$) family of $C^{2+\alpha}$-surfaces with boundary, using Proposition \ref{prop:regular_distance} (possibly decreasing $\rho$ and $\sigma$ depending only on $h_{E_0}$) we can show that for all $r\in [0,\rho]$ the sets 
$$
\Sigma_{r,s}^+:=\{x\in \Omega^{\sigma}\setminus G_0^+[0,s]:\,\dist(x,\Gamma_0^+[0,s])=r\}
$$
and 
$$
\Sigma_{r,s}^-:=\{x\in \Omega^{\sigma}\cap G_0^-[0,s]:\,\dist(x,\Gamma_0^-[0,s])=r\}
$$
are $C^{2+\alpha}$-families in $r$ and $s$ of $C^{2+\alpha}$-surfaces, whose boundaries are on $\{x_3 = \sigma\}$ and a union of $C^{2+\alpha}$-curves homotopic to $\gamma_0.$ Now as in step 1 we construct $C^{2+\alpha}$-surfaces $\Gamma_0^\pm[r,s]$ with boundary (still $C^{2+\alpha}$-regular in $r$ and $s$), first starting from $\gamma_{r,s}^\pm$ satisfying the contact angle condition with $\beta\pm s$, and then extending until we reach $\Sigma_{r,s}^\pm$ such a way that the distance between $\Gamma_0^\pm[r,s]$ and $\Gamma_0^\pm[0,s]$ is $r.$  

Now for any $s\in [0,\sigma]$ and $r\in[0,\rho]$ we denote by $G_0^\pm[r,s]$ the bounded set enclosed by $\p\Omega$ and $\Gamma^\pm[r,s]$. By construction and step 1, $G_0^\pm[r,s]$ satisfies assertions (a) and (b). 

We claim that $\sigma,$ $\rho$ and $\{G^\pm\}$ satisfies the remaining assertions of the lemma. Indeed, $\sigma$ depends only on $\|II_{E_0}\|_\infty,$ $h_{E_0},$  $\rho$ depends only on $\sigma,$ $h_{E_0}$ and $\|II_{E_0}\|_\infty,$ and $G^\pm[\cdot,\cdot]$ admits a parametrization $p^\pm\in C^{2+\alpha,2+\alpha}(([0,\rho]\times[0,\sigma])\times \cl{\sU}),$ which satisfies the assumptions of Definition \ref{def:admissible_sets} (c) of admissible family with $Q=[0,\rho]\times[0,\sigma].$ 
\end{proof}

\begin{corollary}\label{cor:time_foliations}
Let $\beta\in C^{1+\alpha}(\p\Omega),$ $\alpha\in(0,1],$ satisfy \eqref{beta_coercive} and $\{E[t]\}_{t\in[0,T)}$ be an admissible family contact angle $\beta.$ Then for any $T'\in(0,T)$ there exist $\rho\in(0,1)$ and $\sigma\in (0,\eta)$ depending only $\sup_{t\in[0,T']} \|II_{E[t]}\|_\infty$ and $\inf_{t\in[0,T']} h_{E[t]},$ and admissible families $\{G_0^\pm[r,s,a]\}_{(r,s,a)\in[0,\rho]\times[0,\sigma]\times[0,T']}$ such that $G_0^\pm[0,0,a]=E[a]$ and for all $(r,s,a)\in [0,\rho]\times [0,\sigma]\times [0,T']$:

\begin{itemize}[itemsep=3pt]
\item[\rm(a)]  $\dist(\pOmega G_0^\pm[r,s,a], \pOmega E[a])\ge r+s$ and
\begin{align*}
&G_0^-[r,s,a] \subset E[a] \subset G_0^+[r,s,a],\\
&\dist(\pOmega G_0^\pm[r,s,a],\pOmega G_0^\pm[0,s,a])=r,\\
&\dist(\pOmega G_0^\pm[0,s,a],\pOmega E[a])=s;
\end{align*}

\item[\rm(b)] $G_0^\pm[r,s,a]$ is admissible with contact angle $\beta\pm s.$ 

%\item[\rm(c)] for all $r',r''\in [0,\rho/64]$
%$$
%G_0^+[3\rho/16+r',s,a] \subset G_0^+[\rho/2-r'',s,a],\quad 
%G_0^-[3\rho/16+r',s,a] \supset G_0^-[\rho/2-r'',s,a];
%$$
%and 
%$$
%\dist(\pOmega G_0^\pm[\rho,s,a], %\pOmega G_0^\pm[\rho/2-r',s,a]) \ge \rho/64.
%$$ 
\end{itemize}

\end{corollary}

\begin{proof}
By the definition of admissibility, $E[\cdot]$ admits a parametrization $p\in C^{1+\frac{\alpha}{2},2+\alpha}([0,T']\times \cl{\sU})$ for any $T'\in(0,T).$ Therefore, repeating the same arguments of Lemma \ref{lem:foliations} we construct the required family $\{G_0^\pm[r,s,a]\}_{(r,s,a)\in[0,\rho]\times[0,\sigma]\times[0,T']}.$
\end{proof}

Now we study the existence and uniqueness of the mean curvature flow starting from a bounded droplet and its some stability properties. 

\begin{theorem}\label{teo:short_time}
Let $\beta\in C^{1+\alpha}(\p\Omega)$ (for some $\alpha\in(0,1]$) satisfy \eqref{beta_coercive} and $E_0\subset\Omega$ be an admissible set with contact angle $\beta.$ Then there exist a maximal time $T^\dag>0$ and a unique family $\{E[t]\}_{t\in [0,T^\dag)}$ of admissible sets in $\Omega$ such that $E[0]=E_0,$ $E[t]$ is admissible with contact angle $\beta$ and the hypersurfaces $\pOmega E[t]$ flow by mean curvature, i.e.,
\begin{equation}\label{mce_main}
v_{E[t]}(x) = -\kappa_{E[t]}(x) \quad \text{for $t\in[0,T^\dag)$ and $x\in \pOmega E[t],$}
\end{equation}
where $v_{E[t]}$ is the normal velocity of $\pOmega E(t).$ Moreover, for $T\in (0,T^\dag),$ let $\rho\in(0,1),$ $\sigma\in(0,\eta)$ and the families $\{G_0^\pm[r,s,a]\}_{(r,s,a)\in[0,\rho]\times [0,\sigma]\times[0,T']}$ be given by Corollary \ref{cor:time_foliations}. 
Then (possibly decreasing $\rho$ and $\sigma$ slightly, depending only on $\{E(t)\}$) there exist unique admissible families $\{G^\pm[r,s,a,t]\}_{(r,s,a)\in[0,\rho]\times [0,\sigma]\times[0,T'],t\in[a,T]}$ such that
\begin{itemize}
\item $G^\pm[r,s,a,a]=G_0^\pm[r,s,a],$

\item $G^\pm[r,s,a,t]$ is admissible with contact angle $\beta\pm s,$ 

\item 
\begin{equation}\label{forced_curva}
v_{G^\pm[r,s,a,t]}(x) = - \kappa_{G^\pm[r,s,a,t]}(x) \pm s\quad\text{for $t\in (a,T)$ and $x\in \pOmega G^\pm[r,s,a,t].$}
\end{equation}
\end{itemize}
Furthermore, 
\begin{itemize}
\item[\rm(a)] $G^\pm[0,0,a,t]=E[t]$ for all $t\in[a,T];$

\item[\rm(b)] there exists an increasing continuous function $g:[0,+\infty)\to[0,+\infty)$ with $g(0)=0$ such that 
$$
\max_{x\in \pOmega G^\pm[0,s,a,t]}\,\,\dist(x,\pOmega G^\pm[0,0,a,t]) \le g(s)
$$
for all $s\in[0,\sigma],$ $a\in[0,T]$ and $t\in[0,T];$

\item[\rm(c)] there exists $t^*\in (0,\rho/64)$ (independent of $r,s$ and $a$) such that 
\begin{equation}\label{military_conflict}
G_0^+[\rho/2,s,a] \subset G^+[\rho,s,a,a+t']
\quad 
\text{and}\quad 
G_0^-[\rho/2,s,a] \supset G^-[\rho,s,a,a+t']
\end{equation}
for all $t'\in [0,t^*]$ with $a+t'\le T.$
\end{itemize}
 
\end{theorem}

Thus, $\{G^\pm[r,s,a,\cdot]\}$ is a mean curvature flow starting from $G_0^\pm[r,s,a]$ and with forcing $s$ and contact angle $\beta\pm s.$

\begin{proof}
The solvability of \eqref{mce_main} follows from \cite[Theorem B.1]{BKh:2018} and the solvability of \eqref{forced_curva} follows from the well-posedness of \eqref{mce_main} together with the  smooth dependence on the initial datum (see also \cite[Theorem 7.1]{ATW:1993} in the case without boundary). Finally, the assertions (a)-(c) follow  from the smooth dependence of $G^\pm$ on $[r,s,a,t].$
\end{proof}

By Proposition \ref{prop:regular_distance} and the regularity of $G^\pm[r,s,a,\cdot]$ in time,  \eqref{forced_curva} can be rewritten as
\begin{equation}\label{never_ever}
\tfrac{\p}{\p t}\,\sd_{G^\pm[r,s,a,t]}(x) = -\kappa_{G^\pm[r,s,a,t]}(x) \pm s\qquad \text{for $t\in(a,T)$ and $x\in \pOmega  G^\pm[r,s,a,t].$}
\end{equation}

\begin{proposition}\label{prop:time_regular_sdist}
For any $s\in(0,\sigma]$ there exists $\tau_0(s)>0$ such that for any $r\in[0,\rho],$ $a\in[0,T),$ $\tau\in(0,\tau_0)$ and $t\in [a+\tau,T]$ 
\begin{equation}\label{kappa_chappa}
\tfrac{\sd_{G^+[r,s,a,t - \tau]}(x)}{\tau} > -\kappa_{G^+[r,s,a,t]}(x) +\tfrac{s}{2},\quad x\in \pOmega G^+[r,s,a,t],
\end{equation}
and
\begin{equation*} 
\tfrac{\sd_{G^-[r,s,a,t - \tau]}(x)}{\tau} < -\kappa_{G^-[r,s,a,t]}(x) - \tfrac{s}{2},\quad x\in \pOmega G^+[r,s,a,t].
\end{equation*}
\end{proposition}

\begin{proof}
We prove the assertion only for $G^+.$ Let 
$$
g(r,s,a,t,x):=\sd_{G^+[r,s,a,t]}(x),\quad t\in [a,T],\,\,x\in \Omega.
$$
By the $C^2$-regularity of $\Gamma[r,s,a,t]:=\pOmega G^+[r,s,a,t]$ (up to the boundary) as well as its smooth dependence on $r,s,a,t,$ there exists $R_0>0$ such that for any $r\in[0,\rho],$ $s\in[0,\sigma],$ $a\in[0,T],$ $t\in[a,T],$ $x\in \Gamma[r,s,a,t]$ and $y\in \Omega\cap \cl{B_{R_0}(x)}$ the projection $\pi[r,s,a,t,y]$ onto $\cl{\Gamma[r,s,a,t]}$ is a singleton. Note that if $\pi[r,s,a,t,y]\in \Gamma(t),$ then $y-\pi[r,s,a,t,y]$ is parallel to the unit normal $\nu_{\Gamma[r,s,a,t]}(\pi[r,s,a,t,y]).$ In particular, 
$$
g(r,s,a,t,y) = 
\begin{cases}
(y-\pi[r,s,a,t,y]) \cdot \nu_{\Gamma[r,s,a,t]}(\pi[r,s,a,t,y]) &  \text{if $\pi[r,s,a,t,y] \in \Gamma[r,s,a,t],$}\\
|y-\pi[r,s,a,t,y]| & \text{if $\pi[r,s,a,t,y] \in \p\Gamma[r,s,a,t].$}
\end{cases}
$$
Let $p[r,s,a,t,\cdot]:\sU\to\R^3$ be a  parametrization of $\{\Gamma[r,s,a,t]\}$ (smoothly depending on $r,s,a,t$). Then there exists a unique $u_{r,s,a,t,y}\in \cl{\sU}$ such that 
$$
\pi[r,s,a,t,y] = p[r,s,a,t,u_{r,s,a,t,y}],\quad y\in B_{R_0}(x).
$$
The uniqueness of $u_{r,s,a,t,y}$ and the regularity of the diffeomorphism $p$ as well as the implicit function theorem at boundary \cite{Dederick:1913} imply that the map $t\mapsto u_{r,s,a,t,y}$ is continuously differentiable in $t\in[a,T]$ uniformly in $r,s,a,y.$ Thus, $t\mapsto g(r,s,a,t,y)$ is also continuously differentiable in $t\in[a,T]$ and the map $t\mapsto g_t(r,s,a,t,y)$ is uniformly continuous. Then in view of \eqref{never_ever}, for any $s\in(0,\sigma]$ there exists $\tau_0(s)>0$ for which  \eqref{kappa_chappa} holds  for any $r\in[0,\rho],$ $a\in[0,T),$ $\tau\in(0,\tau_0)$ and $t\in [a+\tau,T].$
\end{proof}

As in the standard mean curvature flow of (compact) hypersurfaces without boundary, the smooth mean curvature flow of droplets also  enjoys comparison principles, see also \cite[Proposition B.4]{BKh:2018}.

\begin{theorem}[\textbf{Strong comparison}]\label{teo:comparison}
Let $\{E_1(t)\}_{t\in [0,T^\dag)}$ and $\{E_2(t)\}_{t\in [0,T^\dag)}$ be smooth flows with forcing $s_1$ and $s_2$ and contact angles $\beta_1$ and $\beta_2,$ respectively. Assume that $E_1(0)\prec E_2(0),$ $s_1 \le s_2$ and $\beta_1>\beta_2$ on $\p\Omega.$ Then $E_1(t)\prec E_2(t)$ for all $t\in [0,T^\dag).$
\end{theorem}

\begin{proof}
Let $\bar t\in(0,T^\dag)$ be the first contact time of $\cl{\pOmega E_1(\cdot)}$ and $\cl{\pOmega E_2(\cdot)}.$  By the contact angle condition and the assumption $\beta_1> \beta_2,$ a contact point $x_0$ cannot be on $\p\Omega.$ Therefore, from the inclusion $E_1(\bar t)\subset E_2(\bar t)$ we find $\kappa_{E_1(\bar t)}(x_0) \ge \kappa_{E_2(\bar t)}(x_0),$ and hence, from the evolution equation and the assumption $s_1\le s_2$ we get 
$$
v_{E_1(\bar t)}(\bar t,x_0)-v_{E_2(\bar t)}(\bar t,x_0) =  -\kappa_{E_1(\bar t)}(x_0) + s_1 + \kappa_{E_2(\bar t)}(x_0) -s_2 \le0.
$$
Now using the Hamilton trick (see e.g. \cite[Chapter 2]{Mantegazza:2011}) we conclude that the distance between $\pOmega E_1(t)$ and $\pOmega E_2(t)$ is nondecreasing in $(\bar t-\epsilon,\bar t)$ for small $\epsilon>0.$ In particular, $\cl{\pOmega E_1(\bar t)}\cap \cl{\pOmega E_2(\bar t)}=\emptyset,$ a contradiction.
\end{proof}

\subsection{GMM for mean curvature flow of droplets}

Notice that the capillary Almgren-Taylor-Wang functional \eqref{eq:capillar_ATW} can be rewritten as
\begin{equation*}%\label{new_represent_ATW}
\cF_\beta(E;E_0,\tau) = \cC_\beta(E,\Omega)+\frac{1}{\tau}\int_{E} \sd_{E_0}\,dx-\frac{1}{\tau}\int_{E_0} \sd_{E_0}\,dx.
\end{equation*}

Let us recall some properties of $\cF_\beta$ and its minimizers from \cite{BKh:2018}.

\begin{theorem}\label{teo:propertie_minimizers}
Let $E_0\in BV(\Omega;\{0,1\})$ be bounded, $\tau>0$ and $\beta\in L^\infty(\p\Omega)$ satisfy \eqref{beta_coercive}.

\begin{itemize}
\item[\rm(a)] The functional $\cF_\beta(\cdot;E_0,\tau)$ is $L^1(\Omega)$-lower semicontinuous.

\item[\rm(b)] There exists a minimizer $E_\tau$ of $\cF_\beta(\cdot;E_0,\tau)$ and every minimizer of $\cF_\beta(\cdot;E_0,\tau)$ is bounded.

\item[\rm(c)] There exists a bounded set $E_+$ containing $E_0$ such that for any $F_0\subset E_+$ the minimizer of $\cF_\beta(\cdot;F_0,\tau)$ is a subset of $E_+.$

\item[\rm(d)] There exists $\vartheta\in(0,1/2)$ depending only on $\eta$ such that for any minimizer $E_\tau$ of $\cF_\beta(\cdot;E_0,\tau)$
\begin{equation}\label{L_checkisz_abbaho}
\sup_{x\in\cl{E_\tau \Delta E_0}}\,\d_E(x)\le \tfrac{1}{\vartheta}\,\sqrt\tau.
\end{equation}
Moreover, for any ball $B_r(x)$ centered at $x\in\cl{\Omega},$
$$
P(E_\tau,B_r(x)) \le \tfrac{1}{\vartheta}\,r^2,\quad r>0,
$$
and for any ball $B_r(x)$ centered at $x\in \p E_\tau$
\begin{gather*} 
\vartheta \le \tfrac{|B_r(x)\cap E_\tau|}{|B_r(x)|}\le 1-\vartheta
\qquad \text{and}\qquad 
P(E_\tau,B_r(x)) \ge \vartheta r^2
\end{gather*}
whenever $r\in (0,\vartheta \sqrt\tau).$ In particular, $E_\tau$ can be assumed open and $\cH^2(\p E_\tau\setminus \p^*E_\tau)=0.$

\item[\rm(e)] There exist unique minimal and maximal minimizers $E_{\tau*}$ and $E_\tau^*$ of $\cF_\beta(\cdot;E_0,\tau)$ such that $E_{\tau*}\subset E_\tau\subset E_\tau^*$ for any minimizer $E_\tau.$

\item[\rm(f)] If $F_0\subset E_0,$ then for any minimizers $F_\tau$ and $E_\tau$ of $\cF_\beta(\cdot;F_0,\tau)$ and $\cF_\beta(\cdot;E_0,\tau)$ one has $F_\tau\subset E_{\tau}^*$ and $F_{\tau*}\subset E_\tau,$ where $F_{\tau*}$ and $E_\tau^*$ are the minimal and maximal minimizers of  $\cF_\beta(\cdot;F_0,\tau)$ and $\cF_\beta(\cdot;E_0,\tau)$, respectively.

\item[\rm(g)] Let $\beta$ be $C^1$ on $\p\Omega,$ $E_\tau$ be a  minimizer of $\cF_\beta(\cdot; E_0,\tau)$ and $\Gamma:=\pOmega E_\tau.$ Then by \cite[Theorem 1.5]{DPhM:2017} and the standard regularity theory for minimizers of the presribed curvature functional, $\Gamma$ is a $C^{2+\gamma}$-hypersurface with boundary (for some $\gamma\in (0,1]$) and $\nu_{E_\tau}\cdot \be_3 = -\beta$ on $\p\Gamma.$
\end{itemize}
\end{theorem}

\noindent 
Using the statements (a)-(d) in Theorem \ref{teo:propertie_minimizers} we can establish

\begin{theorem}[\textbf{Existence of GMM \cite{BKh:2018}}]\label{teo:existence_GMM}
For any bounded $E_0\in BV(\Omega;\{0,1\})$ the $GMM(\cF_\beta,E_0)$ is nonempty. Moreover, there exists $C>0$ such that every  $E(\cdot)\in GMM(\cF_\beta,E_0)$ is bounded uniformly in time and satisfies 
$$
|E(s)\Delta E(t)|\le C|t-s|^{1/2},\quad t,s>0.
$$
If, additionally, $|\p E_0|=0,$ then this inequality holds for all $s,t\ge0.$ 
\end{theorem}

\subsection{Comparison of flat flows with truncated balls}

The main result of this section is the following proposition.

\begin{theorem}\label{teo:compare_with_ball}
Let $E_0\subset\Omega$ be a bounded set of finite perimeter, $\beta\in L^\infty(\p\Omega)$ satisfy \eqref{beta_coercive} and $p\in {\Omega}$ with $R:=\dist(p,\pOmega E_0).$ For $\tau>0$ let $\{E(\tau,k)\}$ be flat flows starting from $E_0.$ 
Then for any $\beta_0\in (\|\beta\|_\infty\},1)$  
\begin{align}
\Omega\cap B_{R_0}(p)\subset E_0 \qquad & \Longrightarrow \qquad \Omega\cap B_{\frac{\beta_0 R_0}{16}}(p)\subset E(\tau,k), \label{stay_inside}\\
B_{R_0}(p)\cap E_0 = \emptyset \qquad & \Longrightarrow \qquad B_{\frac{\beta_0 R_0}{16}}(p)\cap E(\tau,k) = \emptyset \label{stay_outside}
\end{align}
whenever $0<\tau < \vartheta_0 R_0^2$ and $0\le k\tau \le \vartheta_0R_0^2,$ where $\vartheta_0\in(0,1)$ is a constant depending only on $\beta_0$ and $\vartheta.$
\end{theorem}

Theorem \ref{teo:compare_with_ball} is a generalization of \cite[Theorem 5.4]{ATW:1993} in the capillary setting. Notice that due to the presence of boundary terms in $\cF_\beta,$ we cannot argue as in the proof of \cite[Theorem 5.4]{ATW:1993}.
We postpone the proof after the following lemma.

\begin{lemma}\label{lem:compare_setg}
Let $E_0,F_0$ be bounded sets of finite perimeter in $\Omega$ with $|\p E_0|=|\p F_0|=0$, $\beta_1,\beta_2\in L^\infty(\p\Omega)$ satisfy \eqref{beta_coercive} and for $\tau>0$ let $E_\tau$ be a minimizer of $\cF_{\beta_1}(\cdot;E_0,\tau)$ and  $F_{\tau*}$ be the minimal minimizer of $\cF_{\beta_2}(\cdot;F_0,\tau).$

\begin{itemize}
\item[\rm(a)] Let $F_0\subseteq E_0$ and $\beta_2\ge \beta_1.$ Then $F_{\tau*}\subseteq E_\tau.$

\item[\rm(b)] Let $E_0\cap F_0=\emptyset$ and $\beta_1+\beta_2\ge0.$ Then $F_{\tau*}\cap E_\tau=\emptyset.$
\end{itemize}

\end{lemma}

\begin{proof}
(a) Summing 
$$
\cF_{\beta_1}(E_\tau;E_0,\tau) \le \cF_{\beta_1}(E_\tau\cup F_{\tau*};E_0,\tau)
\quad\text{and}\quad 
\cF_{\beta_2}(F_{\tau*};F_0,\tau) \le \cF_{\beta_2}(F_{\tau*}\cap E_{\tau};F_0,\tau)
$$
and using 
$$
P(E_\tau,\Omega) + P(F_{\tau*},\Omega) \le P(E_\tau\cup F_{\tau*},\Omega) + P(F_{\tau*}\cap E_\tau,\Omega)
$$
we get 
$$
\frac1\tau\int_{F_{\tau*} \setminus E_\tau} [\sd_{F_0} - \sd_{E_0}]dx + \int_{\p\Omega} [\beta_2-\beta_1]\chi_{F_{\tau*}\setminus E_\tau}d\cH^2\le0.
$$
Since $\sd_{F_0}\ge \sd_{E_0}$ (by assumption $F_0\subset E_0$) and $\beta_2\ge\beta_1$, this inequality holds if and only if $|F_{\tau*}\setminus E_\tau|=0$ (see also \cite[Section 6]{BKh:2018}).

(b) Summing 
$$
\cF_{\beta_1}(E_\tau;E_0,\tau) \le \cF_{\beta_1}(E_\tau\setminus F_{\tau*};E_0,\tau)
\quad\text{and}\quad 
\cF_{\beta_2}(F_{\tau*};F_0,\tau) \le \cF_{\beta_2}(F_{\tau*}\setminus E_{\tau};F_0,\tau)
$$
and using 
$$
P(E_\tau,\Omega) + P(F_{\tau*},\Omega) \le P(E_\tau\cap F_{\tau*}^c,\Omega) + P(F_{\tau*}^c\cup E_\tau,\Omega) = P(E_\tau\setminus F_{\tau*},\Omega) + P(F_{\tau*}\setminus E_\tau,\Omega)
$$
we get 
$$
\frac1\tau\int_{F_{\tau*} \cap E_\tau} [\sd_{F_0} + \sd_{E_0}]dx + \int_{\p\Omega} [\beta_1+\beta_2]\chi_{F_{\tau*}\cap E_\tau}d\cH^2\le0.
$$
Since $E_0\cap F_0=\emptyset,$ we have $\sd_{E_0} + \sd_{F_0}\ge0$ in $\Omega.$ Therefore, recalling $\beta_1+\beta_2\ge0,$ we find that the last inequality holds if and only if $|E_\tau\cap F_{\tau*}|=0.$
\end{proof}

Now we are ready to prove relations \eqref{stay_inside}-\eqref{stay_outside}.

\begin{proof}[Proof of Theorem \ref{teo:compare_with_ball}]

We start by showing \eqref{stay_inside}. Depending on the position of $p$ we distinguish three cases.
\bigskip

{\it Case 1:} $B_{R_0}(p)\subset \Omega.$ 
\bigskip

Take any $B_r(p)\subset \Omega,$ $r>0.$ It is well-known that the unique minimizer of $\cF_{\beta_0}(\cdot; B_r(p),\tau)$ is a (possibly empty) ball $B_\rho(p)$ for some $\rho\ge0.$ By \eqref{L_checkisz_abbaho} 
$$
|r-\rho| \le \frac{\sqrt\tau}{\vartheta} \le \frac{4r}{5}
$$
provided $0<\tau<16\vartheta^2r^2/{25}.$ Thus, $\rho\ge 4r/5.$ Moreover, differentiating the function
$$
s\mapsto \cF_\beta(B_s;B_r,\tau) = 4\pi^2 s^2 + \frac{4\pi^2}{\tau}\int_0^s t^2(t-r)dt
$$
at $s=\rho,$ we obtain 
$$
\frac{r-\rho}{\tau} = \frac{2}{\rho} \le \frac{10}{4r}.
$$
This yields 
$$
2r^2 - 5\tau \le 2\rho r \le \rho^2+r^2,
$$
and thus, 
\begin{equation}\label{mashaha}
\rho^2 \ge r^2 - 5\tau,\quad 0<\tau<\tfrac{16\vartheta^2r^2}{25}.
\end{equation}

Let $\{B_{\rho(\tau,k)}(p)\}_{k\ge0}$ be the flat flows starting from $B_{R_0}(p)$ and associated to $\cF_{\beta_0}.$ By \eqref{mashaha}
\begin{equation}\label{radii_decraese}
\rho(\tau,k)^2 \ge \rho(\tau,k-1)^2 - 5\tau,\quad 0<\tau<\tfrac{16\vartheta^2\rho(\tau,k-1)^2}{25},
\end{equation}
for any $k\ge1$ with $\rho(\tau,k-1)>0,$ where $\rho(\tau,0)=R_0.$ Fix any $k_0\ge1.$ From \eqref{radii_decraese} and the monotonicity of $k\mapsto\rho(\tau,k)$ for any $1\le k\le k_0$  we get 
$$
\rho(\tau,k)^2 \ge R_0^2- 5k\tau,\quad 0<\tau<\tfrac{16\vartheta^2\rho(\tau,k-1)^2}{25}
$$
whenever $\rho(\tau,k-1)>0.$ Thus, if we choose $\tau<4\vartheta^2R_0^2/ 25$ and $5k_0\tau\le 3R_0^2/4 < 5(k_0+1)\tau$ (so that $\rho(\tau,k_0)\ge R_0/2$), then $\rho(\tau,k)\ge R_0/2$ for all $0\le k\le k_0.$
This and comparison lemma \ref{lem:compare_setg} (a) yield 
\begin{equation}\label{step1radius}
B_{R_0/2}\subset B_{\rho(\tau,k)}(p) \subset E(\tau,k),\quad 0<\tau<\tfrac{4\vartheta^2R_0^2}{25},\quad 0\le k\tau \le \tfrac{3R_0^2}{20}.
\end{equation}
\smallskip

{\it Step 2:} $B_{R_0}(p)\setminus\cl{\Omega}\ne\emptyset,$ but $B_{\beta_0 R_0/8}(p)\subset\Omega.$ 
\bigskip

By step 1 (applied with $R_0:=\beta_0R_0/8$)
\begin{equation}\label{step2radius}
B_{\beta_0R_0/16}(p) \subset E(\tau,k),\quad 0<\tau<\tfrac{\vartheta^2\beta_0^2R_0^2}{400},\quad 0\le k\tau < \tfrac{3\beta_0^2R_0^2}{1280}.
\end{equation}
\smallskip

\begin{figure}[htp!]
\includegraphics[width=0.5\textwidth]{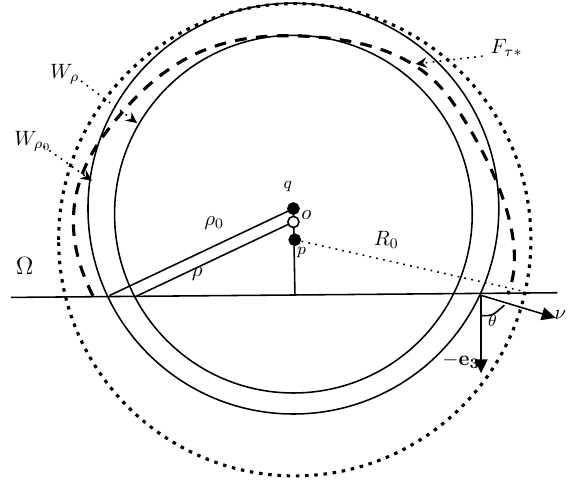}
\caption{Winterbottom shapes.}\label{fig:wulffs_evolving}
\end{figure}

{\it Step 3:} $B_{\beta_0R_0/8}(p)\setminus \cl{\Omega}\ne\emptyset$ i.e., $p\cdot\be_3<\beta_0R_0/8.$ 
\bigskip

Recall that (see e.g. \cite[Chapter 19]{Maggi:2012}, \cite{Kholmatov:2024} and \cite{KSch:2024}) in the isoperimetric inequality 
\begin{equation}\label{winterbottom_defos}
\cC_{\beta_0}(E) = P(E,\Omega) + \int_{\p\Omega} \beta_0\chi_{E}d\cH^2 \ge c_{\beta_0} |E|^{\frac{2}{3}},\quad E\in BV(\Omega;\{0,1\})
\end{equation}
the equality holds if and only if $E$ is a Winterbottom shape with contact angle $\beta_0$, i.e., the truncation of a ball with $\p\Omega$ at angle $\beta_0.$ Thus, if $B_\ell(a)\cap \Omega$ is a Winterbottom shape, then necessarily 
$$
a\cdot \be_3 = \ell \beta_0 = \ell \cos\theta,
$$
where $\theta$ is the angle between $-\be_3$ and the outer unit normal of the ball at its boundary points intersecting $\p\Omega.$ 

Let us denote by $W_r$ the Winterbottom shape whose radius is $r>0$ and center is located on the straight line passing through $p$ and parallel to $\be_3.$ One can readily check that $W_r\subset W_{r'}$ if $r<r'.$

By assumption of step 3, there exists a Winterbottom shape $W_r\subset B_{R_0}(p).$
Let $W_{\rho_0}:=B_{\rho_0}(q)\cap\Omega$ be the largest Winterbottom shape contained in $B_{R_0}(p).$ Then, in the notation of Figure \ref{fig:wulffs_evolving}, 
\begin{equation}\label{largest_winterbottom}
q\cdot \be_n = \rho_{0}\cos\theta,\quad \rho_0 = \tfrac{R_0 + p\cdot \be_3}{1+\cos\theta}.
\end{equation}
Let $F_{\tau*}$ be the minimal minimizer of $\cF_{\beta_0}(\cdot; W_{\rho_0},\tau).$ Then by \eqref{L_checkisz_abbaho}
$$
\sup_{x\in F_0 \Delta F_{\tau*}} \d_{F_0}(x) = 
\sup_{x\in \Omega\cap [B_{\rho_0}(q)\Delta F_{\tau*}]}\,\,||x|-\rho_0| \le \frac{\sqrt\tau}{\vartheta} \le \frac{\rho_0}{5}
$$
provided that $0<\tau<{\vartheta^2\rho_0^2}/{25}$.
Thus,
\begin{equation}\label{nima_bor_orqasida}
B_{4\rho_0/5}(q) \subset  F_{\tau*}.
\end{equation}
Next, let $W_\rho$ and $W_{\bar \rho}$ be the largest Winterbottom shapes contained in $F_{\tau*}$ and in $B_{4\rho_0/5}(q),$ respectively. By \eqref{nima_bor_orqasida}, $\rho\ge \bar\rho.$
Moreover, as in \eqref{largest_winterbottom}
$$
\bar\rho = \tfrac{4\rho_0/5 + \bar o\cdot\be_n}{1+\cos\theta},
$$
where $\bar o\in\Omega$ is the center of $W_{\bar\rho}.$ Thus,
\begin{equation}\label{headliefht}
\rho\ge \bar\rho \ge  \tfrac{4\rho_0}{5(1+\cos\theta)}.
\end{equation}

As in step 1, we would like to estimate $\rho_0-\rho$ from above by a multiple of $1/\rho_0$. If $\rho\ge\rho_0,$ we are done. Otherwise, fix a small $\epsilon\in(0,\rho_0-\rho)$ and consider the Winterbottom shape $W_{\rho+\epsilon}.$ The maximality of $\rho$ implies $W_\rho\subset F_{\tau*}$ and $|W_{\rho+\epsilon}\setminus F_{\tau*}|>0.$ Moreover, by the minimality of $F_{\tau*},$
$$
\cF_{\beta_0}(F_{\tau*},F_0,\tau) \le \cF_{\beta_0}(F_{\tau*}\cup W_{\rho+\epsilon},F_0,\tau),
$$
or equivalently,
\begin{equation}\label{using_minimality}
\cC_{\beta_0}(F_{\tau*}\cup W_{\rho+\epsilon})-\cC_{\beta_0}(F_{\tau*}) \ge\frac{1}{\tau}\int_{W_{\rho+\epsilon}\setminus F_{\tau*}} \d_{F_0}dx.
\end{equation}
Since the  centers of the Winterbottom shapes $W_{\rho_0}$ and $W_\rho$ lie on the same line, 
$$
\d_{F_0} \ge \rho_0-\rho-\epsilon-|[qo_\epsilon]| \quad\text{in $W_{\rho+\epsilon}\setminus F_{\tau*},$}
$$
where $o_\epsilon\in\Omega$ is the center of $W_{\rho+\epsilon}.$ By the contact angle condition, as in \eqref{largest_winterbottom}
$$
|[qo_\epsilon]| = q\cdot\be_n - o_\epsilon\cdot \be_n = (\rho_0 - \rho-\epsilon) \cos\theta.
$$
Thus, 
\begin{equation}\label{signed_dpart}
 \frac{1}{\tau}\int_{W_{\rho+\epsilon}\setminus F_{\tau*}} \d_{F_0}dx\ge \frac{\rho_0-\rho-\epsilon}{\tau}\,(1-\cos\theta)|W_{\rho+\epsilon}\setminus F_{\tau*}|.
\end{equation}
On the other hand, for a.e. $\epsilon$ with $\cH^{2}(\pOmega W_{\rho}\cap \pOmega F_{\tau*})=0$ we have 
$$
\cC_{\beta_0}(F_{\tau*}\cup W_{\rho+\epsilon})-\cC_{\beta_0}(F_{\tau*}) = 
\cC_{\beta_0}(W_{\rho+\epsilon}) - \cC_{\beta_0}(W_{\rho+\epsilon}\cap F_{\tau*}) \le c_{\beta_0} \Big(|W_{\rho+\epsilon}|^{\frac{2}{3}} - |W_{\rho+\epsilon}\cap F_{\tau*}|^{\frac{2}{3}} \Big).
$$
where in the last inequality we used \eqref{winterbottom_defos} and the definition of the Winterbottom shape. Since $|W_{\rho+\epsilon}\setminus F_{\tau*}|\to0$ as $\epsilon\to0^+,$ using the obvious inequality 
$$
(1-t)^\alpha\ge 1 - t,\quad \alpha,t\in(0,1)
$$
we deduce 
\begin{align*}
 |W_{\rho+\epsilon}|^{\frac{2}{3}} - |W_{\rho+\epsilon}\cap F_{\tau*}|^{\frac{2}{3}} = 
 |W_{\rho+\epsilon}|^{\frac{2}{3}} \left( 1- \left[1 - \tfrac{|W_{\rho+\epsilon}\setminus F_{\tau*}|}{|W_{\rho+\epsilon}|} \right]^{\frac{2}{3}}\right) \le \tfrac{|W_{\rho+\epsilon}\setminus F_{\tau*}|}{|W_{\rho+\epsilon}|^{1/3}}.
\end{align*}
Therefore, 
\begin{equation}\label{capillary_estimates}
\cC_{\beta_0}(F_{\tau*}\cup W_{\rho+\epsilon})-\cC_{\beta_0}(F_{\tau*}) \le \tfrac{c_{\beta_0}|W_{\rho+\epsilon}\setminus F_{\tau*}|}{|W_{1}|^{1/3}(\rho+\epsilon)}.
\end{equation}
By the definition of the Winterbottom shape $c_{\beta_0}=\cC_{\beta_0}(W_1)/|W_1|^{2/3}$ and $\beta_0=\cos\theta.$  
Inserting \eqref{signed_dpart} and \eqref{capillary_estimates} in \eqref{using_minimality}, and letting $\epsilon\to0$ we get 
$$
\tfrac{\rho_0-\rho}{\tau} \le \tfrac{\cC_{\beta_0}(W_1)}{|W_1|(1-\beta_0)}\,\tfrac{1}{\rho} \le \tfrac{5\cC_{\beta_0}(W_1) (1+\beta_0)}{4|W_1|(1-\beta_0)}\,\tfrac{1}{\rho_0} =: \tfrac{C_{\beta_0}}{\rho_0},
$$
where in the last inequality we used \eqref{headliefht}. 
As in \eqref{mashaha} this yields 
\begin{equation}\label{new_mashahas}
\rho^2  \ge \rho_0^2- C_{\beta_0}\tau,\quad 0<\tau<\tfrac{\vartheta^2\rho_0^2}{25}. 
\end{equation}

Let the numbers $\rho_0:=\rho(\tau,0)\ge \rho(\tau,1)\ge\ldots>0$ be defined as follows: for each $k\ge1,$ if $\rho(\tau,k)<\rho(\tau,k-1),$ then the $W_{\rho(\tau,k)}\ne\emptyset$ is the maximal Winterbottom shape stying inside the minimal minimizer of $\cF_{\beta_0}(\cdot; W_{\rho(\tau,k-1)},\tau).$
By \eqref{new_mashahas}
$$
\rho(\tau,k)^2 \ge \rho(\tau,0) - C_{\beta_0}k\tau = \rho_0^2-C_{\beta_0}k\tau \ge \frac{\rho_0^2}{4}
$$
provided that $0<\tau<{\vartheta^2\rho^2}/{100}$ and $0\le k\tau\le {3\rho_0^2}/(4C_{\beta_0}).$
 
Consider the family $\{E(\tau,k)\}$ of flat flows starting from $E_0$ and associated to $\cF_\beta.$ Since 
$$
W_{\rho(\tau,0)} \subset \Omega\cap B_{R_0}(p)\subset E(\tau,0)
$$
and $\|\beta\|_\infty\le \beta_0,$ by Lemma \ref{lem:compare_setg} and the definition of $W_{\rho(\tau,1)}$ we have 
$
W_{\rho(\tau,1)}\subset E(\tau,1).
$
By induction 
$$
W_{\rho_0/2}\subset 
W_{\rho(\tau,k)}\subset E(\tau,k),\quad 0<\tau<\tfrac{\vartheta^2\rho_0^2}{100},\quad 0\le k\tau\le \tfrac{3\rho_0^2}{4C_{\beta_0}}.
$$
Let $\{y_k\}$ be the centers of these Winterbottom shapes. As in \eqref{largest_winterbottom} $\beta_0\rho(\tau,k)=y_k\cdot\be_3.$ Then by assumption of the case 3 and the definition $\rho_0$ in \eqref{largest_winterbottom} 
$$
p\cdot \be_3\le \tfrac{\beta_0 R_0}{8} < \tfrac{\beta(1+\beta)}{8}\rho_0 <\tfrac{\beta_0\rho_0}{2}-\tfrac{\beta_0\rho_0}{4} \le \beta_0\rho(\tau,k) - \tfrac{\beta_0\rho_0}{4} = y_k\cdot \be_3 - \tfrac{\beta_0\rho_0}{4}.
$$
This implies $B_{\beta_0\rho_0/4}(p)\subset B_{\rho(\tau,k)}(y_k)$ and therefore, again using the definition $\rho_0$ in \eqref{largest_winterbottom} we find 
\begin{equation}\label{step3radius}
\Omega\cap B_{\frac{\beta_0}{4(1+\beta_0)}R_0} \subset 
E(\tau,k),\quad 0<\tau<\tfrac{\vartheta^2R_0^2}{100(1+\beta_0)},\quad 0\le k\tau\le \tfrac{3R_0^2}{4C_{\beta_0}(1+\beta_0)}.
\end{equation} 
Since $\beta_0\in (0,1),$ \eqref{stay_inside} follows from \eqref{step1radius},  \eqref{step2radius} and \eqref{step3radius}.

To prove \eqref{stay_outside} we repeat the same arguments using Lemma \ref{lem:compare_setg} (b) in place of (a).
\end{proof}

\subsection{Smooth barriers for minimizers of $\cF_\beta$}

The aim of this section is the following analogue of \cite[Lemma 7.3]{ATW:1993}.

\begin{lemma}\label{lem:barrier}
Let $\beta\in C^{1+\alpha}(\p\Omega),$ $\alpha\in (0,1],$ satisfy \eqref{beta_coercive}, $E_0 \in BV(\Omega;\{0,1\})$ be a bounded set, $\tau>0$ and $E_\tau$ be a minimizer of $\cF_\beta(\cdot; E_0,\tau).$ Let $G_0$ and $G_\tau$ be any admissible sets.

\begin{itemize}
\item[\rm(a)] Assume that $E_0\subset G_0,$ $E_\tau \subset G_\tau,$ $G_\tau$ satisfies the contact angle condition with $\beta +s$ for some $s\in (0,\eta)$ and 
\begin{equation}\label{dnaufh7634275}
\tfrac{\sd_{G_0}(x)}{\tau} > -\kappa_{G_\tau}(x),\quad x\in \pOmega G_\tau. 
\end{equation}
Then $E_\tau \prec G_\tau.$

\item[\rm(b)] Assume that $G_0\subset E_0,$ $G_\tau \subset E_\tau,$ $G_\tau$ satisfies the contact angle condition with $\beta - s$ for some $s\in (0,\eta)$ and 
$$
\tfrac{\sd_{G_0}(x)}{\tau} < -\kappa_{G_\tau}(x),\quad x\in \pOmega G_\tau.
$$
Then $G_\tau \prec E_\tau.$
\end{itemize}

\end{lemma}

\begin{proof}
(a) By the regularity of $E_\tau$ (see Theorem \ref{teo:propertie_minimizers} (f)), $\cl{\pOmega E_\tau}$ is a $C^2$-hypersurface with boundary, and hence, by the first variation formula, 
\begin{equation}\label{eular_legreg}
\tfrac{\sd_{E_0}(x)}{\tau} = -\kappa_{E_\tau}(x),\quad x\in \pOmega E_\tau,\quad\text{and}\quad \nu_{E_\tau}(x) \cdot \be_3 = -\beta(x),\quad x\in \p\Omega \cap \cl{\pOmega E_\tau}.
\end{equation}
By contradiction, let there exist $x_0\in \cl{\pOmega E_\tau}\cap \cl{\pOmega G_\tau}.$ By assumption $E_\tau\subset G_\tau$ and the contact angle condition it follows that  $x_0\in\Omega$ and $\kappa_{E_\tau}(x_0)\ge \kappa_{G_\tau}(x_0).$ On the other hand, by assumption $E_0\subset G_0$ and \eqref{compare_trunc_sdist} $\sd_{E_0}(x_0)\ge \sd_{G_0}(x_0),$ and therefore, from \eqref{dnaufh7634275} and the first equality in \eqref{eular_legreg} we obtain
$$
\tfrac{\sd_{G_0}(x_0)}{\tau} \le \tfrac{\sd_{E_0}(x_0)}{\tau} = -\kappa_{E_\tau}(x_0) \le -\kappa_{G_\tau}(x) <\tfrac{\sd_{G_0}(x_0)}{\tau},
$$
a contradiction.

(b) is analogous.
\end{proof}

\section{Proof of Theorem \ref{teo:consistence}} \label{sec:proof_consos}

Let $\{E(t)\}_{t\in [0,T^\dag)}$ be a smooth mean curvature flow starting from $E_0$ and with contact angle $\beta,$ and let $F(\cdot)\in GMM(\cF_\beta,E_0).$ Following \cite[Theorem 7.4]{ATW:1993} we fix any $T\in (0,T^\dag)$ and show 
\begin{equation}\label{shortly_equals}
E(t)=F(t)\quad\text{for any $0 < t < T.$}
\end{equation}
Let $\rho\in(0,1),$ $\sigma\in(0,\eta),$ the smooth flows $\{G^\pm [r,s,a,t]:\,\,(r,s,a)\in[0,\rho]\times [0,\sigma]\times [0,T],\,t\in [a,T]\}$ starting from $\{G_0^\pm[r,s,a]\},$ and $t^*>0$ be given by the second part of Theorem \ref{teo:short_time}. Let $\tau_j\searrow0$ and flat flows $\{F(\tau_j,k)\}_{k\ge0}$ starting from $E_0$ and associated to $\cF_\beta$ satisfy 
\begin{equation}\label{flats_converge}
\lim\limits_{j\to+\infty} |F(\tau_j,\intpart{t/\tau_j})\Delta F(t)| =0\quad\text{for all $t\ge0.$}
\end{equation}

We start with an ancillary technical lemma.
For $s\in(0,\sigma]$ let $\tau_0(s)>0$ be given by Proposition \ref{prop:time_regular_sdist} and for $\beta_0:=\frac{1+\|\beta\|_\infty}{2}$ let $\vartheta_0$ be given by Theorem \ref{teo:compare_with_ball}. We may assume that $\tau_j<\vartheta_0\rho^2/64^2$ for all $j.$ 

\begin{lemma}\label{lem:chaklpakes}
Assume that $t_0\in[0,T)$ and $k_0\in\N_0$ are such that 
\begin{equation}\label{initially_Good}
G_0^-[0,s,t_0] \subset F(\tau_j,k_0) \subset G_0^+[0,s,t_0].
\end{equation}
Then there exists $\bar t\in(0,t^*]$ depending only on $t^*$ and $\rho$ such that 
\begin{equation*}
G^-[0,s,t_0, t_0+k\tau_j] \subset F(\tau_j,k_0+k) \subset G^+[0,s,t_0,t_0+k\tau_j]
\end{equation*}
for all $s\in(0,\sigma],$ $j\ge1$ with $\tau_j\in(0,\tau_0(s))$ and $k=0,1,\ldots, \intpart{\bar t/\tau_j}$ with $t_0+k\tau_j<T.$ 
Moreover, let $t_0+\bar t < T,$ the increasing continuous function $g$ be given by Theorem \ref{teo:short_time} (b) and $\bar\sigma\in(0,\sigma/2)$ be such that $4g(2\bar\sigma)<\sigma.$ Then for any $s\in(0,\bar\sigma)$ there exists $\bar j(s)>1$ such that 
\begin{equation}\label{tiktok_videos}
G_0^-[0,4g(2s), t_0+\bar t] \subset F(\tau_j,k_0+\bar k_j) \subset G_0^+[0,4g(2s), t_0+\bar t]
\end{equation}
whenever $j>\bar j(s),$ where $\bar k_j:=\intpart{\bar t/\tau_j}.$
\end{lemma}

\begin{proof}
By Corollary \ref{cor:time_foliations} (a) 
$$
\dist(\pOmega G_0^\pm[\rho/4,s,t_0],\pOmega G_0^\pm[0,s,t_0]) = \rho/4,
$$ 
and therefore,  by \eqref{initially_Good} 
\begin{equation}\label{init_Better}
G_0^-[\rho/4,s,t_0]\prec G_0^-[0,s,t_0] \subset F(\tau_j,k_0) \subset G_0^+[0,s,t_0]\prec G_0^+[\rho/4,s,t_0]
\end{equation}
and by \eqref{init_Better} $B_{\rho/4}(x)\subset F(\tau_j,k_0) $ if $ x\in G_0^-[\rho/4,s,t_0]$ and 
$B_{\rho/4} (x)\cap F(\tau_j,k_0) =\emptyset $ if $x\in \Omega\setminus G_0^+[\rho/4,s,t_0].$ Therefore, using Theorem \ref{teo:compare_with_ball} (with $R_0=\rho/4$ and $\beta_0:= \frac{1+\|\beta\|_\infty}{2}$) and again \eqref{initially_Good} we obtain
\begin{equation}\label{bolls0987}
\begin{cases}
B_{\frac{\beta_0\rho}{64}} (x)\subset F(\tau_j,k_0+k) &  x\in G_0^-[\rho/4,s,t_0],\\
B_{\frac{\beta_0\rho}{64}} (x)\cap F(\tau_j,k_0+k) =\emptyset & x\in \Omega\setminus G_0^+[\rho/4,s,t_0],
\end{cases}
\quad k=0,1,\ldots,\intpart{t^{**}/\tau_j},
\end{equation}
where 
$$
t^{**}:= \tfrac{\vartheta_0\rho^2}{16}.
$$
By \eqref{bolls0987} and Corollary \ref{cor:time_foliations} (b)
\begin{equation}\label{qoshona}
G_0^-\Big[\tfrac{\rho}{2}, s,t_0\Big]\subset 
G_0^-\Big[\tfrac{\rho}{4}- \tfrac{\beta_0\rho}{64}, s,t_0\Big] \subset F(\tau_j,k_0+k) \subset G_0^+\Big[\tfrac{\rho}{4}- \tfrac{\beta_0\rho}{64}, s,t_0\Big] \subset G_0^+\Big[\tfrac{\rho}{4}, s,t_0\Big]
\end{equation}
for all $0\le k\le \intpart{t^{**}/\tau_j}.$ 
Set 
$$
\bar t:=\min\Big\{t^*,t^{**}\Big\},
$$ 
where $t^*$ is given by Theorem \ref{teo:short_time} (c). Then by \eqref{military_conflict} and \eqref{qoshona} 
\begin{equation}\label{qoshona_new}
G_0^-[\rho, s,t_0+k\tau_j] \subset F(\tau_j,k_0+k) \subset G_0^+[\rho, s,t_0+ k\tau_j],\quad k=0,1,\ldots, \intpart{\bar t/\tau_j},
\end{equation}
with $t_0+k\tau_j<T.$ 
We claim for such $k$  and $j\ge1$ with $\tau_j\in(0,\tau_0(s))$
\begin{equation*}
G^-[0,s,t_0,t_0+k\tau_j] \subset F(\tau_j,k_0+k) \subset G^+[0,s,t_0, t_0+k\tau_j].
\end{equation*}
Indeed,  let 
$$
\bar r:=\inf\Big\{r\in [0,\rho]:\,\,\, F(\tau_j,k_0+k) \subset G^+[r,s,t_0,t_0+k\tau_j]\quad k=0,1,\ldots,\intpart{\bar t/\tau_j},\,\,t_0+k\tau_j<T\Big\}.
$$
By \eqref{qoshona_new} the infimum is taken over a nonempty set. By contradiction, assume that $\bar r>0.$ In view of the continuity of $G^+[r,s,t_0,t_0+k\tau_j]$ at $r=\bar r,$ there exists the smallest integer $k\le \intpart{\bar t/\tau_j}$ (clearly, $k>0$ by \eqref{init_Better}) such that 
\begin{equation}\label{chandui_jahon098}
\cl{\pOmega F(\tau_j,k_0+k)} \cap \cl{\pOmega G^+[\bar r,s,t_0,t_0+k\tau_j]} \ne\emptyset.
\end{equation}
By the minimality of $k\ge1$ 
$$
F(\tau_j,k_0+k-1)\subset G^+[\bar r,s,t_0,t_0+(k-1)\tau_j],\qquad 
F(\tau_j,k_0+k)\subset G^+[\bar r,s,t_0,t_0 + k\tau_j].
$$
Moreover, by construction $G^+[\bar r,s,t_0,t_0+k\tau_j]$ satisfies the contact angle condition with $\beta+s$ and by Proposition \ref{prop:time_regular_sdist} applied with $\tau=\tau_j\in(0,\tau_0(s))$
$$
\tfrac{\sd_{G^+[\bar r,s,t_0, t_0+(k-1)\tau_j]}(x)}{\tau_j} > -\kappa_{G^+[\bar r,s,t_0,t_0 + k\tau_j]}(x) +\tfrac{s}{2},\quad x\in \pOmega G^+[\bar r,s,t_0,t_0 + k\tau_j].
$$
However, in view of Lemma \ref{lem:barrier} (a), these properties imply $F(\tau_j,k_0+k)\prec G^+[\bar r,s,t_0,t_0+k\tau_j],$ which contradicts to \eqref{chandui_jahon098}. Thus, $\bar r=0.$  Analogous contradiction argument based on Lemma \ref{lem:barrier} (b) shows $G^-[0,s,t_0,t_0+k\tau_j] \subset F(\tau_j,k_0+k)$ for all $0\le k\le \intpart{\bar t/\tau_j}.$

Finally, let us prove \eqref{tiktok_videos}. Recall that by construction $G_0^-[0,2s,t_0]\prec G_0^-[0,s,t_0]$ and $G_0^+[0,s,t_0]\prec G_0^+[0,2s,t_0],$ therefore, by the strong comparison principle (Theorem \ref{teo:comparison}) $G^-[0,2s,t_0,t]\prec G^-[0,s,t_0,t]$ and $G^+[0,s,t_0,t]\prec G^+[0,2s,t_0,t]$ for all $t\in[t_0,T].$ Now the continuity of $G^\pm[0,s,t_0,t]$ on its parameters we could find $\bar j=\bar j(s)>1$ such that for all $j>\bar j$
\begin{multline}\label{mahszrte}
G^-[0,2s,t_0,t_0+\bar t] \prec G^-[0,s,t_0,t_0+\bar k_j\tau_j] \\
\subset F(\tau_j,\bar k_j) \subset G^+[0,s,t_0,t_0 +\bar k_j\tau_j]\prec G^+[0,2s,t_0,t_0+\bar t].
\end{multline}
By the definition of $g,$
\begin{equation}\label{copsos}
\max\limits_{x\in \pOmega G^\pm[0,2s,t_0,t_0+ \bar t]}\,\, \dist(x, \pOmega E(t_0+ \bar t)) \le g(2s)
\end{equation}
and therefore, by construction in Corollary \ref{cor:time_foliations} (a) 
$$
\dist(\pOmega G_0^\pm[0,4g(2s),t_0+\bar t],\pOmega E(t_0+ \bar t)) = 4g(2s)>0.
$$
Combining this with \eqref{copsos} and the construction of $G_0^\pm$ we deduce 
$$
G_0^-[0,4g(2s),t_0+\bar t] \prec G^-[0,2s,t_0,t_0+\bar t] 
\quad\text{and}\quad 
G^+[0,2s,t_0,t_0+\bar t]
\prec 
G_0^+[0,4g(2s),t_0+\bar t].
$$
These inclusions together with \eqref{mahszrte} imply \eqref{tiktok_videos}.
\end{proof}

Now we are ready to prove the equality \eqref{shortly_equals}. Let $\bar t$ be given by Lemma \ref{lem:chaklpakes},
$$
N:=\intpart{T/\bar t} + 1
$$
and let $\sigma_0\in(0,\sigma/16)$ be such that the numbers
$$
\sigma_l=4g(2\sigma_{l-1}),\quad l=1,\ldots,N,
$$
satisfy $\sigma_l\in (0,\sigma/16).$ By the monotonicity and continuity of $g$ together with $g(0)=0,$ such choice of $\sigma_0$ is possible.

Fix any $s\in(0,\sigma_0)$ and let 
$$
a_0(s):=s,\quad a_l(s):=4g(2a_{l-1}(s)),\quad l=1,\ldots,N.
$$
Note that $a_l(s)\in (0,\sigma_l).$ In particular, the numbers $\bar j_l^s:=\bar j(a_l(s)),$ given by the last assertion of Lemma \ref{lem:chaklpakes}, are well-defined. Let also 
$$
\tilde j_l^s:=\max\{j\ge1:\,\, \tau_j\notin (0,\tau_0(a_l(s)))\}
$$
and
$$
\bar j_s:= 1 + \max_{l=0,\ldots,N}\,\max \{\bar j_l^s,\tilde j_l^s\}.
$$

By Corollary \ref{cor:time_foliations} (a)
$$
G_0^-[0,s,0]\subset E(0)=E_0=F(\tau_j,0) \subset G_0^+[0,s,0]
$$
for all $j>\bar j_s.$
Therefore, by Lemma \ref{lem:chaklpakes} applied with $k_0=0$ and $t_0=0$ we find 
$$
G^-[0,s,0,k\tau_j]\subset F(\tau_j,k) \subset G^+[0,s,0,k\tau_j],\quad k=0,1,\ldots,\bar k_j,
$$
where $\bar k_j:=\intpart{\bar t/\tau_j}.$ Moreover, since $s\in (0,\sigma_0,)$ by the last assertion of Lemma \ref{lem:chaklpakes} 
$$
G_0^-[0,a_1(s),\bar t]\subset F(\tau_j,\bar k_j) \subset G_0^+[0,a_1(s),\bar t] 
$$
for all $j\ge \bar j_s.$ Hence, we can reapply Lemma \ref{lem:chaklpakes} with $s:=a_1(s),$ $t_0=\bar t$ and $k_0=\bar k_j,$ to find 
$$
G^-[0,a_1(s),0,\bar t + k\tau_j]\subset F(\tau_j,\bar k_j+k) \subset G^+[0,a_1(s),0,\bar t+ k\tau_j],\quad k=0,1,\ldots,\bar k_j.
$$
In particular, since $j>\bar j_s> \bar j(a_1(s)),$ again by the last assertion of  Lemma \ref{lem:chaklpakes}  we deduce 
$$
G_0^-[0,a_2(s),2\bar t]\subset F(\tau_j,2\bar k_j) \subset G_0^+[0,a_2(s),2\bar t]. 
$$
Repeating this argument at most $N$ times, for all $j\ge \bar j_s$ we find 
\begin{equation}\label{nadomatlar671}
G^-[0,a_l(s),0,l\bar t + k\tau_j]\subset F(\tau_j,l\bar k_j+k) \subset G^+[0,a_l(s),0,l\bar t + k\tau_j],\quad k=0,1,\ldots,\bar k_j
\end{equation}
whenever $l=0,\ldots,N$ and $l\bar t + k\tau_j\le T.$ 

Now take any $t\in(0,T),$ and let $l:=\intpart{t/\bar t}$  and $k=\intpart{t/\tau_j} - l\bar k_j$ so that $l\bar k_j + k = \intpart{t/\tau_j}.$ By means of $l$ and $k,$ as well as the definition of $\bar k_j$ we represent \eqref{nadomatlar671} as 
\begin{multline}\label{headlight09}
G^-\Big[0,a_l(s),0,l\bar t + \tau_j\intpart{\tfrac{t}{\tau_j}} - l\tau_j\intpart{\tfrac{\bar t}{\tau_j}}\Big]\\
\subset 
F\Big(\tau_j, \intpart{\tfrac{t}{\tau_j}}\Big) 
\subset 
G^+\Big[0,a_l(s),0,l\bar t + \tau_j\intpart{\tfrac{t}{\tau_j}} - l\tau_j\intpart{\tfrac{\bar t}{\tau_j}}\Big]
\end{multline}
for all $j>\bar j_s.$ Since 
$$
\lim\limits_{j\to+\infty} \Big(l\bar t + \tau_j\intpart{\tfrac{t}{\tau_j}} - l\tau_j\intpart{\tfrac{\bar t}{\tau_j}}\Big) = t,
$$
by the continuous dependence of $G^\pm$ on its parameters, as well as the convergence \eqref{flats_converge} of the flat flows, letting $j\to+\infty$ in \eqref{headlight09} we obtain
\begin{equation}\label{endiyonimga}
G^-[0,a_l(s),0,t]\subset F(t) \subset G^+[0,a_l(s),0,t],
\end{equation}
where due to the $L^1$-convergence the inclusions in \eqref{flats_converge} are up to some negligible sets. Now we let $s\to0^+$ and recalling that $a_l(s)\to0$ (by the continuity of $g$ and assumption $g(0)=0$), from \eqref{endiyonimga} we deduce 
\begin{equation*}
G^-[0,0,0,t]\subset F(t) \subset G^+[0,0,0,t].
\end{equation*}
Then by Theorem \ref{teo:short_time} (a) 
$$
F(t)=G^\pm[0,0,0,t] = E(t).
$$

\end{document}